\documentclass[12pt]{amsart}

\usepackage[T1]{fontenc}
\usepackage{lmodern}

\usepackage{mathtools,amsthm,amssymb}

\usepackage[sort&compress,numbers]{natbib}

\newcommand{\N}{{\mathbb N}}
\newcommand{\Z}{{\mathbb Z}}
\newcommand{\R}{{\mathbb R}}

\newcommand{\eps}{\varepsilon}

\newcommand{\bX}{{\boldsymbol{X}}}
\newcommand{\bG}{{\boldsymbol{G}}}
\newcommand{\bU}{{\boldsymbol{U}}}
\newcommand{\tX}{{\widetilde{X}}}
\newcommand{\tV}{{\widetilde{V}}}
\newcommand{\fX}{{\mathfrak{X}}}

\newcommand{\cS}{\widehat{{\mathcal{S}}}}
\newcommand{\tS}{\widehat{S}}

\newcommand{\ee}{{\mathrm{e}}}
\newcommand{\cc}{{\mathfrak{c}}}

\newcommand{\Sd}{{\mathbb S}^{\mathrm{det}}}
\newcommand{\Sr}{{\mathbb S}^{\mathrm{ran}}}
\newcommand{\Sb}{{\mathbb S}^{\mathrm{bit}}}
 
\newcommand{\Lcon}{\gamma}
\newcommand{\E}{\operatorname{E}}
\newcommand{\Lip}{\operatorname{Lip}}
\newcommand{\cost}{\operatorname{cost}}
\newcommand{\cb}{\operatorname{comp}^{\mathrm{bit}}}
\newcommand{\cd}{\operatorname{comp}^{\mathrm{det}}}
\newcommand{\cra}{\operatorname{comp}^{\mathrm{ran}}}

\newcommand{\class}{\mathrm{c}}
\newcommand{\bit}{\mathrm{bit}}
\newcommand{\stwo}{\mathrm{Bbit}}

\newcommand{\Var}{\operatorname{Var}}
\newcommand{\rbit}{\operatorname{rbit}}

\theoremstyle{definition}
\newtheorem{defn}{Definition}
\theoremstyle{plain}
\newtheorem{thm}[defn]{Theorem}
\newtheorem{lem}[defn]{Lemma}
\newtheorem{cor}[defn]{Corollary}

\theoremstyle{remark}
\newtheorem{rem}[defn]{Remark}
\newtheorem{exmp}[defn]{Example}

\begin{document}

\title[Random Bit Multilevel for SDEs]
{Random Bit Multilevel Algorithms for Stochastic Differential Equations}

\author[Giles]
{Michael B.\ Giles}
\address{Mathematical Institute\\
University of Oxford\\
Oxford OX2 6GG\\
England} 
\email{mike.giles@maths.ox.ac.uk} 

\author[Hefter]
{Mario Hefter}
\address{Fachbereich Mathematik\\
Technische Universit\"at Kaisers\-lautern\\
Postfach 3049\\
67653 Kaiserslautern\\
Germany}
\email{\{hefter,lmayer,ritter\}@mathematik.uni-kl.de}

\author[Mayer]
{Lukas Mayer}

\author[Ritter]
{Klaus Ritter}

\begin{abstract}
We study the approximation of expectations $\E(f(X))$
for solutions $X$ of SDEs and functionals $f \colon C([0,1],\R^r) \to \R$
by means of restricted Monte Carlo algorithms that may only 
use random bits instead of random numbers.
We consider the worst case setting for functionals $f$ from the 
Lipschitz class w.r.t.\ the supremum norm.
We construct a random bit multilevel Euler algorithm and establish
upper bounds for its error and cost.
Furthermore, we derive matching lower bounds, up to
a logarithmic factor, that are valid for all
random bit
Monte Carlo algorithms, and we show that, for the given
quadrature problem,
random bit
Monte Carlo algorithms are at least almost as powerful
as general randomized algorithms.
\end{abstract}



\date{January 16, 2019}

\maketitle

\section{Introduction}

We study the approximation of expectations $\E(f(X))$,
where $X=(X(t))_{t\in [0,1]}$ is 
the $r$-dimensional solution of an autonomous SDE 
with deterministic initial value and 
Lipschitz continuous drift and diffusion coefficients,
driven by a $d$-dimensional Brownian motion.
Moreover, $f \in \Lip_1$, i.e.,
\[
f \colon C([0,1],\R^r)  \to \R
\]
is Lipschitz continuous with
respect to the supremum norm with Lipschitz constant at most one.

We consider randomized (Monte Carlo) algorithms
that are only allowed to use random bits instead of 
random numbers. By assumption, all other operations 
(arithmetic operations, evaluations of elementary functions, and 
oracle calls to evaluate $f$ as well as the drift and diffusion
coefficients of the SDE)
are performed exactly.
Algorithms of this type are called
random bit
Monte Carlo algorithms, and the approximation of expectations
by algorithms of this type is called random bit quadrature.
Due to the path-dependence of $f$ the approximation of
$\E(f(X))$ is an infinite-dimensional
quadrature problem.

In a worst case setting algorithms $A$ are compared according to their 
worst case error $\ee(A,\Lip_1)$ and their worst case cost
$\cost(A,\Lip_1)$ on the class $\Lip_1$.
The cost takes into account, in particular, 
the number of random bits used 
and the information cost, i.e., the cost for the evaluations of
$f$. 
For the latter, we suppose that any $f \in \Lip_1$
can be evaluated at any piecewise linear function $x \in C([0,1],\R^r)$
with 
equidistant breakpoints $0,1/2^\ell, \dots, 1$ at cost $2^\ell+1$ for 
any $\ell \in \N_0$.
See Sections \ref{seccomp}, \ref{ls1}, and \ref{s53} for details.

The main contribution of this paper is the construction of a random 
bit multilevel Euler algorithm $A_\eps^\stwo$ that is almost optimal.
First of all, we have the following upper bound:
There exists a constant $c>0$ such that
\[
\ee\bigl(A_\eps^\stwo,\Lip_1\bigr)
\leq c \cdot \eps
\]
and
\begin{align}\label{costbound}
\cost\bigl(A_\eps^\stwo,\Lip_1\bigr)
\leq
c \cdot \eps^{-2} \cdot (\ln(\eps^{-1}))^3
\end{align}
for every $\eps \in {]0,1/2[}$, see Theorem~\ref{theo3}.
This result coincides with the best known upper bound for 
general randomized algorithms, which are achieved by the classical 
multilevel Euler algorithm.
See \citet{G15} for a survey on multilevel algorithms.

An important ingredient for the construction of $A_\eps^\stwo$
is Bakhvalov's trick: A small number of independent
random variables, each uniformly distributed on $\{1,\dots, 2^q\}$,
yields a much larger number of pairwise
independent random variables with the same
uniform distribution.
The number of random bits used by the 
algorithm $A_\eps^{\stwo}$ is of the order 
$\eps^{-2} \cdot (\ln(\eps^{-1}))^{5/2}$, see the proof of 
Theorem~\ref{theo3}, and it can be reduced further to 
$\eps^{-2} \cdot (\ln(\eps^{-1}))^{2} \cdot \ln(\ln(\eps^{-1}))$, 
as outlined in Remark~\ref{cr2}.
Since the upper bound in \eqref{costbound} is sharp,
the number of random bits is asymptotically negligible 
compared to the overall cost of $A_\eps^\stwo$. 

Secondly, the algorithm $A_\eps^\stwo$ is optimal, up to logarithmic 
factors. Under a slightly stronger smoothness assumption as well as a 
non-degeneracy assumption on the diffusion coefficient of the SDE,
which in particular exclude pathological cases yielding a deterministic 
solution $X$, the following holds true:
There exist constants $c,\eps_0>0$ such that
\[
\cost\bigl(A,\Lip_1\bigr) \geq c \cdot \eps^{-2}
\]
for every
random bit
Monte Carlo algorithm $A$ and
for every $\eps \in {]0,\eps_0]}$ such that
\[
\ee\bigl(A,\Lip_1\bigr) \leq \eps.
\]
Actually, there are two variants of this result, both of which hold
for a much broader class of algorithms.
In the first variant, the evaluation of $f$ is allowed at arbitrary 
points $x \in C([0,1],\R^r)$ at cost one, while the
number of random bits is taken into account as before, 
see Theorem~\ref{t23}. 
In the second variant, which is due to \citet[Thm.~11]{CDMGR09},
roughly speaking, any kind of randomness is allowed for free, 
but the cost model with evaluations of $f$ only at piecewise
linear functions is kept.
We do not know whether random bits are as powerful as random numbers
for the quadrature problem under investigation, 
but the upper and lower bounds imply that
random bits are at least almost as powerful as random numbers.

This work is a continuation of \citet{GHMR17}.
Two approximation problems involving random bits have been considered
there: Random bit quadrature, as in the present paper, and
random bit approximation of probability measures,
which is closely related to quantization.
Furthermore, two classes of processes have been considered there:
Gaussian processes and solution processes of SDEs.
For three of the possible four combinations results have
already been obtained in \citet{GHMR17}.
The combination of random bit quadrature for SDEs, which remained open,
is addressed in the present paper.

This work is partially motivated by
reconfigurable architectures like field programmable gate arrays (FPGAs).
These devices allow users to choose
the precision of each individual operation on a bit level and provide a 
generator for random bits.
In the setting and analysis of the present
paper we take into account the latter fact, while we 
ignore all finite precision issues for arithmetic operations.
We refer to
\citet{6924076,bookbits}
for the construction and 
for extensive tests of a
finite precision multilevel algorithm for FPGAs with applications
in computational finance. For an error analysis of the Euler scheme
for SDEs in a finite precision arithmetic we refer to \citet{O16}.
A complete analysis on a bit level for the arithmetic operations, 
the random number generator, and the oracle is not yet available.

In most of the papers on randomized algorithms
for continuous problems, uniformly distributed random numbers 
from $[0,1]$ are assumed to be available.
Random bit
Monte Carlo algorithms are studied 
for the classical, finite-dimensional quadrature problem
to approximate $\int_{[0,1]^d} f(x)\,\mathrm dx$ 
in, e.g., \citet{GYW06,MR2076605,N85,N88,N01,TW92,YH08}.
See \citet{NP04} for a related approach to integral equations. 

In \citet{MR2076605}, random bit quadrature with respect to the uniform 
distribution on $[0,1]^d$ and Sobolev and H\"older classes of 
functions $f \colon [0,1]^d \to \R$ are considered. It is shown that
random bit
Monte Carlo algorithms are as powerful as 
general randomized algorithms,
and a very small number of random bits suffice 
to achieve asymptotic optimality. The proofs of these results are based
on a reduction of the quadrature problem to a summation problem
and on Bakhvalov's trick. 

In contrast to, e.g., \citet{MR2076605}, we do not derive lower
bounds that separately take into account the number of random bits
and the
information cost.
Instead, we establish lower bounds in terms of the overall cost,
which includes, in particular, the sum of both of these quantities.
A general framework to study
Monte Carlo
algorithms that only have access to generators for an arbitrary, but fixed
set of probability distributions
has recently been introduced in \citet{H18}.
Within the latter framework the information cost and the cost
associated to the calls of the available random number generators
are naturally studied separately. 

This paper is organized as follows. In Section~\ref{seccomp} we formulate 
the computational problem. Section~\ref{Euler:schemes} is devoted to the 
definition and strong error analysis of a random bit Euler scheme.
In Section~\ref{sec5} we present the construction of the random bit 
multilevel Euler algorithm with corresponding error and cost bounds.
Lower bounds for
random bit
Monte Carlo methods are derived in 
Section~\ref{bounds}.
In Appendix~\ref{app} we present Bakhvalov's trick in a form that 
fits to our needs.

\section{The Computational Problem}\label{seccomp}

Let $r,d \in \N$. Henceforth we use $|\cdot|$ to denote the 
Euclidean norm, and we consider the corresponding
supremum norm $\|\cdot\|$ on $C([0,1],\R^r)$.

We consider an autonomous system 
\[
\phantom{\qquad\quad t \in [0,1]}
\mathrm{d}X(t) = a(X(t))\,\mathrm{d}t + b(X(t))\,\mathrm{d}W(t), 
\qquad\quad t \in [0,1],
\]
of SDEs with a deterministic initial value
\[
X(0) = x_0 \in \R^r
\]
and a $d$-dimensional Brownian motion $W$, and with
Lipschitz-continuous drift and diffusion coefficients
$a \colon \R^r \to \R^r$ and $b \colon \R^r \to \R^{r \times d}$,
respectively. We study the approximation of
\begin{equation}\label{g20}
S(f) = \E(f(X))
\end{equation}
for functionals $f \colon C([0,1],\R^r) \to \R$
that are Lipschitz continuous with Lipschitz constant at most one, i.e.,
\[
|f(x) - f(y)| \leq \|x - y\|
\]
for all $x,y \in C([0,1],\R^r)$.
The class of all such Lipschitz functionals $f$ is denoted by $\Lip_1$.

We employ the real-number model, i.e., we assume
that algorithms can perform comparisons, arithmetic operations
on real numbers, and the evaluation of elementary functions
at unit cost. Furthermore, we assume that the drift coefficient $a$ and
the diffusion coefficient $b$ can be evaluated at each
point $x \in \R^r$ at unit cost.
Moreover, we suppose that any functional $f \in \Lip_1$
can be evaluated at any piecewise linear function with 
breakpoints $k/2^\ell$ for $k=0,\dots,2^\ell$ at cost $2^\ell+1$ for 
any $\ell \in \N_0$.

Monte Carlo algorithms have access to a random number generator at cost 
one per call. Here, we distinguish two cases, namely algorithms with
access to random numbers from $[0,1]$, 
and algorithms with access to random bits only, 
which we refer to as
random bit
Monte Carlo algorithms.

The cost, $\cost(A,f)$, of applying the
Monte Carlo algorithm
$A$ to the functional $f$ is defined as the sum of the cost
associated to every instruction that is carried out.
Observe that under appropriate measurability assumptions 
$\cost(A,f)$ is a random quantity.
We define the worst case cost of $A$ on
the class $\Lip_1$ as
\[
\cost(A,\Lip_1)
=
\sup_{f \in \Lip_1} \E(\cost(A,f)).
\]

Since the output $A(f)$ of a
Monte Carlo
algorithm $A$ applied to an input 
functional $f$ from $\Lip_1$ is a random quantity, too,
its error is defined by
\[
\ee(A,f)
=
\bigl(\E|S(f) - A(f)|^2\bigr)^{1/2}
\]
under appropriate measurability assumptions.
Accordingly, the worst case error on the class $\Lip_1$ is defined by
\[
\ee(A,\Lip_1)
=
\sup_{f \in \Lip_1} \ee(A,f).
\]
See Section \ref{ls1} for a rigorous definition of a more general
cost model that only takes into account the information
cost concerning $f$ and the number of random bits used.

\section{Euler Schemes}\label{Euler:schemes}

The key ingredient for the construction and analysis of random bit 
multilevel algorithms, see Section~\ref{sec5},
is a random bit Euler scheme and strong error bounds thereof,
which will be presented in Sections~\ref{scheme1} 
and~\ref{scheme1:analysis}.

Let
\[
t_k = t_{k,m} = k/m,
\]
where $m \in \N$ and $k=0,\ldots,m$.
A corresponding Euler scheme is given by
\begin{align*}
X_m(t_{0,m}) &= x_0,\\
X_m(t_{k,m}) &= X_m(t_{k-1,m}) + m^{-1} \cdot a(X_m(t_{k-1,m})) + 
b(X_m(t_{k-1,m})) \cdot V_{k,m}
\end{align*}
with suitable random vectors $V_{k,m}$ to be defined below.
The multilevel approach relies on a coupling
of $X_m$ with an even number $m\in\N$ of steps to an Euler scheme 
\begin{align*}
\tX_{m/2}(t_{0,m/2})
&=
x_0,\\
\tX_{m/2}(t_{k,m/2})
&=
\tX_{m/2}(t_{k-1,m/2}) 
\begin{aligned}[t]
&+ 
(m/2)^{-1} \cdot a\bigl(\tX_{m/2}(t_{k-1,m/2})\bigr)\\
&+ b\bigl(\tX_{m/2}(t_{k-1,m/2})\bigr) 
\cdot \tV_{k,m/2}
\end{aligned}
\end{align*}
with $m/2$ steps.
Suitable random vectors $\tV_{k,m/2}$ will be defined below.

In order to approximate $X$ at any point 
$t \in [0,1]$ we extend $X_m(t_0),\ldots,X_m(t_m)$ and
$\tX_{m/2}(t_0),\ldots,\tX_{m/2}(t_{m/2})$ by linear 
interpolation onto the subintervals $\left]t_{k-1},t_k\right[$.

\subsection{The Classical Euler Scheme}\label{cleu}

In the vast majority of papers, $X_m$ is based
on the Brownian increments
\[
V_{k,m} = W(t_{k,m}) - W(t_{k-1,m}),
\]
which are naturally coupled via 
\begin{equation}\label{g10}
\tV_{k,m/2} = V_{2k,m} + V_{2k-1,m} =
W(t_{k,m/2}) - W(t_{k-1,m/2}).
\end{equation}
Obviously, $\tV_{k,m/2} = V_{k,m/2}$, which is 
convenient in the analysis of multilevel algorithms.
For the corresponding classical Euler scheme we use the notation 
$X_m^{\class}$ and $\tX_{m/2}^{\class}$, and likewise we use
$V_{k,m}^{\class}$ and $\tV_{k,m/2}^{\class}$ for the
corresponding increments.

\subsection{A Random Bit Euler Scheme}\label{scheme1}

In the present paper we study an Euler scheme that only uses
random bits instead of random numbers from $[0,1]$.
This excludes
the use of Brownian increments.
At first we discuss the approximation of standard
normal distributions based on random bits.

Let $\Phi$ denote the distribution function of $N(0,1)$ 
with inverse function $\Phi^{-1}$,
and let $Y \sim N(0,1)$.
We introduce the rounding function
\begin{align*}
T^{(q)} \colon {[0,1[} \to D^{(q)}, 
\quad x \mapsto \frac{\lfloor{2^q x}\rfloor}{2^q} + 2^{-(q+1)},
\end{align*}
where $q \in \N$ and
\[
D^{(q)} = \{\sum_{i=1}^q b_i \cdot 2^{-i} + 2^{-(q+1)} : 
b_i \in \{0,1\} \ \text{for} \ i=1,\ldots,q\}.
\]
Then
\[
Y^{(q)} = \Phi^{-1} \circ T^{(q)} \circ \Phi(Y)
\]
serves as a canonical approximation of $Y$.

Observe that $T^{(q)} \circ \Phi(Y)$ is uniformly distributed on $D^{(q)}$.
Consequently, $q$ random bits suffice to simulate the distribution of 
$Y^{(q)}$. 
Further properties of $Y^{(q)}$ have been established in
\citet{GHMR17}, see also Remark~\ref{rem:prop} below.
For a standard normally distributed random vector $Y$ an
approximation $Y^{(q)}$ is obtained
by applying $\Phi^{-1} \circ T^{(q)} \circ \Phi$ to each 
of the components of $Y$ separately.

We use this approximation 
in a straightforward way, i.e., we study
a random bit Euler scheme with
\begin{align}\label{eq10}
V_{k,m} = m^{-1/2} \cdot \bigl( m^{1/2} \cdot (W(t_{k,m}) -
W(t_{k-1,m}))\bigr)^{(q)}.
\end{align}
A suitable coupling is easily achieved by
\begin{align}\label{eq2}
\tV_{k,m/2} = V_{2k,m} + V_{2k-1,m},
\end{align}
cf.\ \eqref{g10}. 
To indicate the dependence of this coupled Euler scheme on the bit
number $q$ we use the notation $X_{m,q}^\bit$ and
$\tX_{m/2,q}^\bit$,
and likewise $V^\bit_{k,m,q}$
and $\tV^\bit_{k,m/2,q}$
for the approximations of the Brownian increments. Proper relations
between the 
number $q$ of bits and the number $m$ of Euler steps will 
be presented in Section~\ref{sec5}.

The simulation of the joint distribution of $X_{m,q}^\bit$ and
$\tX_{m/2,q}^\bit$ requires $d \cdot m \cdot q$ random bits.
We stress that
the distributions of $\tV^\bit_{k,m/2,q}$ and
$V^\bit_{k,m/2,q}$ do not coincide,
and therefore
\begin{align}\label{eqend100}
\E(f(X_{m/2}))
\neq
\E(f(\tX_{m/2}))
\end{align}
in general.
This introduces an additional
bias term in the multilevel analysis, cf.\
\citet[Thm.~6.1]{MSS15}.

\subsection{Strong Error Analysis}\label{scheme1:analysis}

It is well known that there exists a constant $c > 0$ such that 
\begin{equation}\label{g5}
\Bigl(\E \|X - X_m^\class\|^2\Bigr)^{1/2}
\leq
c \cdot m^{-1/2} \cdot (\ln(m+1))^{1/2}
\end{equation}
for every $m \in \N$. Thus we will provide an upper bound for the
difference between the random bit Euler scheme $X_{m,q}^\bit$
and the classical Euler scheme $X_m^\class$.
The difference between $\tX_{m/2,q}^\bit$ and
$\tX_{m/2}^\class$ may be treated in the same way.

\begin{rem}\label{rem:prop}
We gather some properties of the random vectors $V_{k,m}^\class$
and $V^\bit_{k,m}$ and of the scheme $X^\bit_{m,q}$.
\begin{itemize}
\item[(a)] 
We have independence of 
$(V^\bit_{k,m,q})_{k=1,\dots,m}$.
\item[(b)] 
There exists a constant $c > 0$ such that 
\begin{align*}
\Bigl(\E\bigl|V_{k,m}^\class-V^\bit_{k,m,q}\bigr|^2\Bigr)^{1/2}
\leq
c \cdot m^{-1/2} \cdot 2^{-q/2} \cdot q^{-1/2}
\end{align*}
and
\begin{align*}
\E\bigl(V^\bit_{k,m,q}\bigr) = 0
\end{align*}
for all $m \in \mathbb{N}$, $k=1,\ldots,m$, and $q \in \N$.
Furthermore, we have
\begin{align*}
\sup_{m \in \mathbb{N}} \sup_{k=1,\ldots,m} \sup_{q \in \N} 
\bigl(m^{1/2} \cdot 
\bigl(\E\bigl|V^\bit_{k,m,q}\bigr|^r\bigr)^{1/r}\bigr) < \infty
\end{align*}
for all $r \geq 1$. See \citet[Thm.~1]{GHMR17}. 
\item[(c)] 
We have $\E \|X^\bit_{m,q}\|^2 < \infty$
for every $m \in \mathbb{N}$ and $q \in \mathbb{N}$.
\end{itemize}
\end{rem}

In the sequel, we assume that the Lipschitz constants of the drift
coefficient $a$ and of the diffusion coefficient $b$ are bounded by 
$\Lcon$.

\begin{lem}\label{l:euler_vs_rbit-euler}
There exists a constant $c > 0$ such that for all $m,q \in \N$ 
we have
\begin{align*}
\max_{k=0,\ldots,m} 
\Bigl(\E \bigl|X_m^\class(t_k) - X^\bit_{m,q}(t_k)\bigr|^2\Bigr)^{1/2}
\leq
c \cdot 2^{-q/2} \cdot q^{-1/2}.
\end{align*}
\end{lem}

\begin{proof}
This proof follows the standard analysis for the classical Euler 
scheme.

For $k=0,\ldots,m-1$ we have
\begin{align*}
X_m^\class(t_{k+1}) - X^\bit_{m,q}(t_{k+1})
=
\xi + \zeta,
\end{align*}
where
\begin{align*}
\xi
=
X_m^\class(t_k) - X^\bit_{m,q}(t_k) + m^{-1} \cdot 
\bigl(a(X_m^\class(t_k)) - a(X^\bit_{m,q}(t_k))\bigr)
\end{align*}
and
\begin{align*}
\zeta
= b(X_m^\class(t_k)) \cdot V_{k+1,m}^\class - 
b(X^\bit_{m,q}(t_k)) \cdot V^\bit_{k+1,m,q}.
\end{align*}
For any pair of components $\xi_i$ and $\zeta_i$ of $\xi$ and $\zeta$, 
respectively, we have
\begin{align*}
\E \bigl(\xi_i \cdot \zeta_i\bigr)
=
0,
\end{align*}
due to properties (a)--(c) from Remark~\ref{rem:prop}. It follows 
that
\begin{align*}
\E\bigl|X_m^\class(t_{k+1}) - X^\bit_{m,q}(t_{k+1})\bigr|^2
=
\E\bigl(|\xi|^2\bigr) + \E\bigl(|\zeta|^2\bigr).
\end{align*}

The Lipschitz continuity of $a$ yields
\begin{equation}\label{g1}
\bigl(\E|\xi|^2\bigr)^{1/2}
\leq
(1+\Lcon/m) \cdot 
\Bigl(\E\bigl|X_m^\class(t_k) - X^\bit_{m,q}(t_k)\bigr|^2\Bigr)^{1/2}.
\end{equation}
Moreover,
\begin{align*}
\E(|\zeta|^2)
&\leq
2 \cdot \Bigl(\E \bigl|b(X_m^\class(t_k)) \cdot 
\bigl(V_{k+1,m}^\class - V^\bit_{k+1,m,q}\bigr)\bigr|^2\\
&
\phantom{\leq} \mbox{}
+ \E\bigl|\bigl(b(X_m^\class(t_k)) - b\bigl(X^\bit_{m,q}(t_k)\bigr)\bigr) 
\cdot V^\bit_{k+1,m,q}\bigr|^2\Bigr).
\end{align*}
Due to the independence of
$(V^\class_{k,m})_{k=1,\dots,m}$
we have
\begin{align*}
\E\bigl|b(X_m^\class(t_k)) \cdot 
\bigl(V_{k+1,m}^\class - V^\bit_{k+1,m,q}\bigr)\bigr|^2
\leq
\E|b(X_m^\class(t_k))|^2 \cdot 
\E\bigl|V_{k+1,m}^\class - V^\bit_{k+1,m,q}\bigr|^2.
\end{align*}
Using
\begin{align*}
\E\bigl|b(X_m^\class(t_k))\bigr|^2
\leq
2 \cdot \Lcon^2 \cdot \E|X_m^\class(t_k)|^2 + 2 \cdot |b(0)|^2
\end{align*}
together with property (b) from Remark~\ref{rem:prop} we obtain a 
constant $c_1 > 0$ such that
\begin{align*}
\E\bigl|b(X_m^\class(t_k)) \cdot 
\bigl(V_{k+1,m}^\class - V^\bit_{k+1,m,q}\bigr)\bigr|^2
\leq c_1 \cdot m^{-1} \cdot 2^{-q} \cdot q^{-1}
\end{align*}
for all $m,q \in \N$ and $k=0,\ldots,m-1$.
Property (a) from Remark~\ref{rem:prop} also gives
\begin{align*}
\E\bigl|\bigl(b(X_m^\class(t_k)) - b\bigl(X^\bit_{m,q}(t_k)\bigr)\bigr) 
\cdot V^\bit_{k+1,m,q}\bigr|^2
\leq
\E\bigl|b(X_m^\class(t_k)) - b\bigl(X^\bit_{m,q}(t_k)\bigr)\bigr|^2 
\cdot \E\bigl|V^\bit_{k+1,m,q}\bigr|^2.
\end{align*}
Moreover,
\begin{align*}
\E\bigl|b(X_m^\class(t_k)) - b\bigl(X^\bit_{m,q}(t_k)\bigr)\bigr|^2
\leq
\Lcon^2 \cdot \E \bigl|X_m^\class(t_k) - X^\bit_{m,q}(t_k)\bigr|^2,
\end{align*}
and property (b) from Remark~\ref{rem:prop} yields the
existence of a constant $c_2 > 0$ such that
\begin{align*}
\E\bigl|\bigl(b(X_m^\class(t_k)) - b\bigl(X^\bit_{m,q}(t_k)\bigr)\bigr) 
\cdot V^\bit_{k+1,m,q}\bigr|^2
\leq
c_2 \cdot m^{-1} \cdot \E \bigl|X_m^\class(t_k) - X^\bit_{m,q}(t_k)\bigr|^2
\end{align*}
for all $m,q \in \N$ and $k=0,\ldots,m-1$.
It follows that
\begin{equation}\label{g2}
\E(|\zeta|^2)
\leq c_3/m \cdot \left( 2^{-q} \cdot q^{-1} + 
\E \bigl|X_m^\class(t_k) - X^\bit_{m,q}(t_k)\bigr|^2 \right)
\end{equation}
with $c_3 = 2 \cdot \max(c_1,c_2)$.
Combining \eqref{g1} and \eqref{g2} we
get the existence of a constant $c > 0$ such that
\begin{align}\label{eq:proof}
\E\bigl|X_m^\class(t_{k+1}) - X^\bit_{m,q}(t_{k+1})\bigr|^2
\leq
(1+c/m) \cdot \E\bigl|X_m^\class(t_k) - X^\bit_{m,q}(t_k)\bigr|^2 + 
c/m \cdot 2^{-q} \cdot q^{-1}
\end{align}
for all $m,q \in \N$ and $k=0,\ldots,m-1$.

A discrete Gronwall-inequality or a straightforward computation yields
\begin{align*}
&\E \bigl|X_m^\class(t_{k+1}) - X^\bit_{m,q}(t_{k+1})\bigr|^2\\
&\qquad\leq
(1 + c/m)^{k+1} \cdot \E \bigl|X_m^\class(t_0) - X^\bit_{m,q}(t_0)\bigr|^2 + 
\sum_{j=0}^k (1 + c/m)^j \cdot c/m \cdot 2^{-q} \cdot q^{-1}\\
&\qquad\leq
(1 + c/m)^m \cdot c/m \cdot (k+1) \cdot 2^{-q} \cdot q^{-1}
\end{align*}
with $c$ according to \eqref{eq:proof},
and hereby the statement for $X_m^\class - X^\bit_{m,q}$ follows.
\end{proof}

\begin{lem}\label{l:euler_vs_rbit-euler:sup}
There exists a constant $c > 0$ such that for all $m,q \in \N$ 
we have
\begin{align*}
\Bigl(\E 
\bigl\|X_m^\class - X^\bit_{m,q}\bigr\|^2\Bigr)^{1/2}
\leq c \cdot 2^{-q/2} \cdot q^{-1/2}.
\end{align*}
\end{lem}

\begin{proof}
This proof uses standard martingale arguments to exchange the maximum 
with the expectation in the error bound of 
Lemma~\ref{l:euler_vs_rbit-euler}.

It suffices to consider the error at the point $t_k$,
since $X_m^\class$ and $X^\bit_{m,q}$ are piecewise linear with
breakpoints $t_k$.
For $k=1,\ldots,m$ we have 
\begin{align}\label{eq:proof:1}
X_m^\class(t_k) - X^\bit_{m,q}(t_k)
&=
m^{-1} \cdot \sum_{\ell=1}^k 
\Bigl(a(X_m^\class(t_{\ell-1})) - a(X^\bit_{m,q}(t_{\ell-1}))\Bigr)\\
&\qquad\mbox{}+  
\sum_{\ell=1}^k \Bigl(b(X_m^\class(t_{\ell-1})) \cdot 
V_{\ell,m}^\class - b(X^\bit_{m,q}(t_{\ell-1})) \cdot
V^\bit_{\ell,m,q}\Bigr),\notag
\end{align}
and therefore
\[
\E \max_{k=0,\ldots,m} \bigl|X_m^\class(t_k) - X^\bit_{m,q}(t_k)\bigr|^2
\leq
2 \cdot \E \max_{k=1,\ldots,m} \bigl| Z_k \bigr|^2
+ 2 \cdot \E \max_{k=1,\ldots,m} \bigl| R_k \bigr|^2
\]
with
\[
Z_k =
m^{-1} \cdot \sum_{\ell=1}^k 
\Bigl(a(X_m^\class(t_{\ell-1})) - a(X^\bit_{m,q}(t_{\ell-1}))\Bigr)
\]
and
\[
R_k =
\sum_{\ell=1}^k 
\Bigl(b(X_m^\class(t_{\ell-1})) \cdot V_{\ell,m}^\class 
- b(X^\bit_{m,q}(t_{\ell-1})) \cdot V^\bit_{\ell,m,q}\Bigr).
\]

For the drift term we have
\begin{align}\label{g3}
\E \max_{k=1,\ldots,m} \bigl|Z_k\bigr|^2
&\leq
\E \max_{k=1,\ldots,m} \frac{k}{m^2} \sum_{\ell=1}^k 
\bigl|a(X_m^\class(t_{\ell-1})) - a(X^\bit_{m,q}(t_{\ell-1}))\bigr|^2\\
&\leq
\frac{\Lcon^2}{m} \cdot 
\sum_{\ell=1}^m \E 
\bigl|X_m^\class(t_{\ell-1}) - X^\bit_{m,q}(t_{\ell-1})\bigr|^2
\notag \\
&\leq
\Lcon^2 \cdot \max_{k=0,\ldots,m} 
\E \, \bigl|X_m^\class(t_k) - X^\bit_{m,q}(t_k)\bigr|^2.
\notag
\end{align}
For the diffusion term, we note that $(R_k)_{k=1,\ldots,m}$ 
is a martingale due to 
properties (b) and (c) from Remark~\ref{rem:prop}.
Consequently, the Doob maximal inequality yields
\begin{align*}
\E \max_{k=1,\ldots,m} |R_k|^2
\leq
4 \cdot \E |R_m|^2.
\end{align*}
Use \eqref{eq:proof:1} and \eqref{g3} to obtain 
\begin{align*}
\E |R_m|^2
&= \E \bigl|X_m^\class(t_m) - X^\bit_{m,q}(t_m) - Z_m \bigr|^2\\
&\leq
2 \cdot \E \bigl|X_m^\class(t_m) - X^\bit_{m,q}(t_m)\bigr|^2 + 
2 \cdot \E \bigl|Z_m\bigr|^2\\
&\leq
2 \cdot (1 + \Lcon^2) \cdot \max_{k=0,\ldots,m}
\E \bigl|X_m^\class(t_k) - X^\bit_{m,q}(t_k)\bigr|^2.
\end{align*}
We conclude that
\[
\E \max_{k=0,\ldots,m} \bigl|X_m^\class(t_k) - X^\bit_{m,q}(t_k)\bigr|^2
\leq
18\cdot(1+\Lcon^2) \cdot \max_{k=0,\ldots,m}
\E \bigl|X_m^\class(t_k) - X^\bit_{m,q}(t_k)\bigr|^2.
\]
Apply Lemma~\ref{l:euler_vs_rbit-euler} to derive the statement
for $X_m^\class-X^\bit_{m,q}$.
\end{proof}

\begin{rem}
Let $\Delta$ denote the Wasserstein distance of order two
on the space of all Borel probability measures on a separable Banach 
space. For any such measure $\mu$ the asymptotics of 
\begin{align*}
\rbit(\mu,p) = 
\inf\{ \Delta(\mu, \nu) \colon 
\nu\text{ uniform distribution with support of size } 2^p\},
\qquad
p\in\N,
\end{align*}
have been recently studied in \citet{C18,GHMR17,XB17,Berger2018,
2018arXiv180907891B}.
Combining Lemma~\ref{l:euler_vs_rbit-euler:sup} with 
\eqref{g5} we get the existence of a constant $c>0$ such that
\begin{align}\label{ineq}
\rbit(\mu,p)
\leq
c\cdot\min_{d\cdot m\cdot q\leq p}\left(m^{-1/2} \cdot (\ln(m+1))^{1/2}
+ 2^{-q/2} \cdot q^{-1/2}\right)
\end{align}
for all $p\in\N$,
where $\mu$ is the distribution of the solution $X$ on the 
Banach space $C([0,1],\R^r)$. It can be shown that the right hand 
side of \eqref{ineq} is of the order $p^{-1/2}\cdot \ln(p)$.
In the scalar case, i.e., $r=d=1$,
and under a slightly stronger smoothness assumption on the coefficients 
$a,b$ as well as a non-degeneracy assumption on the diffusion 
coefficients $b$ of the SDE, which in particular exclude pathological 
cases yielding a deterministic solution,
it has been shown in \citet[Thm.~4]{GHMR17} that
$\rbit(\mu,p)$
is of the order $p^{-1/2}$ for the Banach space $L_2[0,1]$.
Hence the upper bound obtained via \eqref{ineq} is sharp, up to 
logarithmic factors, but matching upper and lower bounds seem to be
unknown in this case.
\end{rem}

Following the proofs of Lemma~\ref{l:euler_vs_rbit-euler}
and Lemma~\ref{l:euler_vs_rbit-euler:sup}, one may establish
analogous results for the difference
$\tX_{m/2}^\class - \tX^\bit_{m/2,q}$.
We only formulate the analogue to Lemma~\ref{l:euler_vs_rbit-euler:sup}.

\begin{lem}\label{lem3}
There exists a constant $c > 0$ such that for 
every $q \in \N$ and every even $m\in \N$ 
we have
\begin{align*}
\Bigl(\E 
\bigl\|\tX_{m/2}^\class - \tX^\bit_{m/2,q}\bigr\|^2\Bigr)^{1/2}
\leq c \cdot 2^{-q/2} \cdot q^{-1/2}.
\end{align*}
\end{lem}

Lemma \ref{l:euler_vs_rbit-euler:sup}
and Lemma \ref{lem3}, together with \eqref{g5}, will be used to
control the variances and the bias of multilevel algorithms in the
following section. 
Concerning the variances, a different approach is provided in 
\citet[Sec.~3.1]{BN17}.

\section{Multilevel Euler Algorithms}\label{sec5}

At first we give a general description of a 
multilevel algorithm based on either of the two Euler
schemes in Section~\ref{Euler:schemes}.
Let $L \in \N$ be the maximum level that is used by this
algorithm.
On every level $\ell=1,\ldots,L$ the algorithm involves
a fine approximation $X_{2^\ell}$ and a coarse
approximation $\tX_{2^{\ell-1}}$ according to
Section~\ref{Euler:schemes}. 
On level $\ell=0$ we only need the fine
approximation $X_{2^0}$, but we set
$\tX_{2^{-1}} = 0$
for notational convenience. Furthermore,
let $N=(N_0,\ldots,N_L)\in\N^{L+1}$ be the vector of
replication numbers on the levels $\ell=0,\ldots,L$.

To define the multilevel algorithm we consider a family 
\[
\bX_{L,N} =
\bigl((X_{2^\ell,i},
\tX_{2^{\ell-1},i})\bigr)_{\ell=0,\ldots,L, i=1,\ldots,N_\ell}
\]
of random elements
$(X_{2^\ell,i},\tX_{2^{\ell-1},i})$ 
such that
\[
\bigl(X_{2^\ell,i},\tX_{2^{\ell-1},i}\bigr)
\stackrel{\mathrm{d}}{=}
\bigl(X_{2^\ell},\tX_{2^{\ell-1}}\bigr)
\]
for 
$\ell=0,\ldots,L$ and $i=1,\ldots,N_\ell$.
The joint distribution of 
$\bX_{L,N}$
will be specified later.
The multilevel Euler algorithm finally reads as
\[
A_{L,N}(f)
=
\frac{1}{N_0} \cdot \sum_{i=1}^{N_0} f\bigl(
X_{2^0,i}\bigr)
+ \sum_{\ell=1}^L \frac{1}{N_\ell} \cdot \sum_{i=1}^{N_\ell}
\bigl(f\bigl(X_{2^\ell,i}\bigr)
- f\bigl(\tX_{2^{\ell-1},i}\bigr)\bigr).
\]

Subsequently we study three different variants of this algorithm.
In the first two variants $\bX_{L,N}$ is independent,
i.e., the family $\bX_{L,N}$ of random variables $(X_{2^\ell,i},
\tX_{2^{\ell-1},i})$ is independent,
which is a standard assumption for multilevel algorithms.
In the first case we use the classical Euler scheme with 
normally distributed increments.
In the other two cases we consider
random bit
Monte Carlo algorithms.
In the second case we simply employ 
random bit approximations of the normally distributed increments.
In the third case we apply Bakhvalov's trick in order to 
reduce the number of random bits; consequently, 
$\bX_{L,N}$
is no longer independent.

\subsection{The Classical Multilevel Algorithm}\label{classicml}

In this section 
$\bX_{L,N}$
is assumed to be independent and 
\[
\bigl(X_{2^\ell},\tX_{2^{\ell-1}}\bigr)
\stackrel{\mathrm{d}}{=}
\bigl(X_{2^\ell}^\class,\tX_{2^{\ell-1}}^\class\bigr)
\]
for $\ell=0,\dots,L$, 
where $\tX_{2^{-1}}^\class = 0$ for notational convenience.
See Section~\ref{cleu}.
Hereby we get the classical multilevel Euler algorithm, which 
we denote by $A_{L,N}^\class$.

We estimate the cost of $A_{L,N}^\class$.
For that purpose we consider
the cost on level $\ell$.
The cost for the arithmetic operations, the evaluations of $a,b$ and $f$
as well as $\Phi^{-1}$
on level $\ell$ is, up to a multiplicative constant, given by 
$N_\ell \cdot 2^\ell$.
The number $r_\ell$ of calls to the random number generator on level 
$\ell$ 
is given by
\begin{align*}
r_\ell=N_\ell\cdot 2^\ell \cdot d.
\end{align*}
Consequently, there exists a constant $c > 1$
such that for all 
$L$ and $N$ the multilevel Euler algorithm satisfies
\begin{align*}
c^{-1} \cdot \sum_{\ell=0}^L 2^\ell\cdot N_\ell
\leq
\cost\bigl(A_{L,N}^\class,\Lip_1\bigr)
\leq
c \cdot \sum_{\ell=0}^L 2^\ell\cdot N_\ell.
\end{align*}

For $\eps \in {]0,1/2[}$ we consider the
algorithm 
\[
A_\eps^\class = A_{L(\eps),N(\eps)}^\class
\]
with maximal level
\[
L=
L(\eps)
=
\bigl\lceil\log_2(\eps^{-2}) + 
\log_2(\log_2(\eps^{-2}))\bigr\rceil 
\]
and with replication numbers
\[
N_\ell=
N_\ell(\eps)
=
\bigl\lceil(L+1) \cdot 2^{-\ell} \cdot \max(\ell,1) \cdot \eps^{-2}
\bigr\rceil
\]
for $\ell=0,\ldots,L$.

The following result is known, see, e.g.,
\citet[Rem.\ 8]{CDMGR09}.
For convenience of the reader we present a proof.

\begin{thm}\label{t:ML-EM}
There exists a constant $c > 1$ such that
the multilevel Euler algorithm $A_\eps^\class$
satisfies
\begin{align*}
\ee(A_\eps^\class,\Lip_1) \leq c \cdot \eps
\end{align*}
and
\begin{align*}
c^{-1} \cdot \eps^{-2} \cdot 
(\ln(\eps^{-1}))^3
\leq
\cost(A_\eps^\class,\Lip_1) \leq c \cdot \eps^{-2} \cdot 
(\ln(\eps^{-1}))^3
\end{align*}
for every $\eps \in {]0,1/2[}$.
\end{thm}

\begin{proof}
We show that both the squared bias and the variance of $A_\eps^\class$
are bounded by $\eps^2$, up to a constant.
Use \eqref{g5} to obtain a constant $c_1>0$ such that
\[
\E\|X-X_{2^\ell}^\class\|^2
\leq
c_1 \cdot 2^{-\ell} \cdot \max(\ell,1)
\]
for every $\ell \geq 0$ and $\E\|X_{2^0}\|^2 \leq c_1$.
Hence we get
\[
\left|S(f) - \E ( A_{L,N}^\class (f)) \right|^2=
\left|\E(f(X)) - \E ( f(X_{2^L}^\class)) \right|^2 \leq
c_1 \cdot 2^{-L} \cdot L
\]
for all $L$ and $N$ and for every $f \in \Lip_1$.
With $L=L(\eps)$ this upper bound is of the order $\eps^{2}$,
as claimed. For the variance we get
\[
\Var(A_{L,N}^\class(f))
\leq
6c_1 \cdot \sum_{\ell=0}^L \frac{\max(\ell,1)}{2^{\ell} \cdot N_\ell}
\]
for all $L$ and $N$ and for every $f \in \Lip_1$.
With $N = N(\eps)$
and $L = L(\eps)$
this upper bound is of the order $\eps^{2}$, too.

To derive the cost bounds 
for $A_\eps^\class$ it remains to
observe that there exists a constant $c_2>1$ such that
\begin{equation}\label{g4}
c_2^{-1} \cdot \eps^{-2} \cdot 
(\ln(\eps^{-1}))^3
\leq
\sum_{\ell=0}^{L(\eps)}  N_\ell(\eps) \cdot 2^\ell
\leq c_2 \cdot \eps^{-2} \cdot 
(\ln(\eps^{-1}))^3
\end{equation}
for all $\eps$.
\end{proof}

\subsection{A Random Bit Multilevel Algorithm}\label{section1}

In this section we consider a multilevel Euler algorithm
that is based on random bits.
As before, we assume that 
$\bX_{L,N}$ 
is independent,
but now we take
\[
\bigl(X_{2^\ell},\tX_{2^{\ell-1}}\bigr)
\stackrel{\mathrm{d}}{=}
\bigl(X^\bit_{2^\ell,q},\tX^\bit_{2^{\ell-1},q}\bigr)
\]
for $\ell=0,\dots,L$ and $q \in \N$,
where $\tX^\bit_{2^{-1},q} = 0$ for notational convenience.
See Section~\ref{scheme1}.
The resulting algorithm is denoted by $A^\bit_{L,N,q}$.

A different construction of a
random bit multilevel algorithm, which is based on the coupling
\eqref{eq2} and which ensures that
\begin{equation}\label{eqend}
\tX_{2^{\ell-1}}
\stackrel{\mathrm{d}}{=}
X_{2^{\ell-1}}
\end{equation}
is satisfied, is
studied in \citet{BN17}.
Recall that \eqref{eqend} is violated in our construction,
except in trivial cases, see \eqref{eqend100}.

The cost of $A^\bit_{L,N,q}$ can be bounded as in Section~\ref{classicml}.
In fact, the number $r_\ell$ of calls to the random number generator on 
level $\ell$ is given by
\begin{align*}
r_\ell=N_\ell \cdot 2^\ell \cdot d\cdot q.
\end{align*}
Furthermore, there exists a constant $c>1$ such that for all $L$, $N$, 
and $q$ the multilevel Euler algorithm satisfies
\begin{align*}
c^{-1} \cdot q\cdot \sum_{\ell=0}^L 2^\ell\cdot N_\ell
\leq
\cost\bigl(A^\bit_{L,N,q},\Lip_1\bigr)
\leq
c \cdot q\cdot \sum_{\ell=0}^L 2^\ell\cdot N_\ell.
\end{align*}

For $\eps \in {]0,1/2[}$ we choose the bit number
\[
q(\eps) = L(\eps)
\]
and $L(\eps)$ as well as $N(\eps)$ as in Section~\ref{classicml}
to obtain the algorithm 
\[
A^\bit_\eps = A^\bit_{L(\eps),N(\eps),q(\eps)}.
\]

\begin{thm}\label{theo2}
There exists a constant $c>1$ such that 
the random bit multilevel Euler algorithm $A^\bit_\eps$ satisfies
\begin{align*}
\ee\bigl(A^\bit_\eps,\Lip_1\bigr)
\leq c \cdot \eps
\end{align*}
and
\begin{align*}
c^{-1} \cdot \eps^{-2} \cdot (\ln(\eps^{-1}))^4
\leq
\cost\bigl(A^\bit_\eps,\Lip_1\bigr)
\leq
c \cdot \eps^{-2} \cdot (\ln(\eps^{-1}))^4
\end{align*}
for every $\eps \in {]0,1/2[}$.
\end{thm}

\begin{proof}
At first we establish the error bound.
Due to Theorem~\ref{t:ML-EM} it suffices to show
the existence of a constant $c>0$ such that
\begin{align*}
\sup_{f \in \Lip_1}
\Bigl(\E\bigl|
A_\eps^\class(f) - A^\bit_\eps(f)\bigr|^2\Bigr)^{1/2}
\leq
c \cdot \eps
\end{align*}
for every $\eps \in {]0,1/2[}$.
For $f \in \Lip_1$ we have
\[
\Bigl(\E\bigl|A_{L,N}^\class(f) - A^\bit_{L,N,q}(f)\bigr|^2\Bigr)^{1/2}\\
\leq
\sum_{\ell=0}^L 
\Bigl(\E
\bigl\|X_{2^\ell}^\class - X^\bit_{2^\ell,q} \bigr\|^2
\Bigr)^{1/2}
+ 
\sum_{\ell=1}^L
\Bigl(\E
\bigl\|\tX_{2^{\ell-1}}^\class -
\tX^\bit_{2^{\ell-1},q}\bigr\|^2\Bigr)^{1/2}.
\]
Apply Lemma~\ref{l:euler_vs_rbit-euler:sup}
as well as Lemma~\ref{lem3}
to obtain the existence of a
constant $c_1>0$ such that
\begin{align*}
\Bigl(\E\bigl|A_{L,N}^\class(f) - A^\bit_{L,N,q}(f)\bigr|^2\Bigr)^{1/2}
\leq
c_1 \cdot L \cdot 2^{-q/2} \cdot q^{-1/2}
\end{align*}
for all $L$, $q$, and $N$ and every $f \in \Lip_1$.
With $q=L=L(\eps)$ this upper bound is of the 
order $\eps$.

The cost bounds follow 
immediately from \eqref{g4}.
\end{proof}

\subsection{A Random Bit Multilevel Algorithm Based on Bakhvalov's Trick}\label{secbah}

In this section we consider a variant of the 
random bit multilevel algorithm $A^\bit_{L,N,q}$
from Section~\ref{section1},
which is based on Bakhvalov's trick, see \citet{MR0172463}
and also \citet{MR2076605} and the references therein.
In our case this trick yields $n^2$ pairwise independent random
variables, each of which is uniformly distributed on $D^{(q)}$, from 
$2n$ independent random variables, each of which is uniformly distributed 
on $D^{(q)}$, see Lemma~\ref{bahvalov} in Appendix~\ref{app}.
Of course, the same is true in the $d$-dimensional situation, i.e., 
with $D^{(q)}$ replaced by $(D^{(q)})^d$.

We construct a random bit multilevel Euler algorithm
$A_{L,N,q}^{\stwo}$ that achieves 
the following properties
at a reduced number of random bits compared 
to $A^\bit_{L,N,q}$:
\begin{itemize}
\item[(i)] The family $\bX_{L,N}$ is pairwise independent.
\item[(ii)] For every $\ell=0,\ldots,L$ we have
\[
\bigl(X_{2^\ell},\tX_{2^{\ell-1}}\bigr)
\stackrel{\mathrm{d}}{=}
\bigl(X^\bit_{2^\ell,q},\tX^\bit_{2^{\ell-1},q}\bigr)
\]
as in Section~\ref{section1}.
\end{itemize}

These properties already ensure that the expectations and the variances of
$A^\bit_{L,N,q}(f)$ and $A_{L,N,q}^{\stwo}(f)$
coincide for every $f\in \Lip_1$, so that
 $\ee(A^\bit_{L,N,q},f)=\ee(A_{L,N,q}^{\stwo},f)$.
In particular, 
\begin{align}\label{error}
\ee\bigl(A^\bit_{L,N,q},\Lip_1\bigr)
=\ee\bigl(A_{L,N,q}^{\stwo},\Lip_1\bigr).
\end{align}

We describe the construction of $A_{L,N,q}^{\stwo}$ 
or, equivalently, the distribution of $\bX_{L,N}$
in detail.
Let
\begin{align*}
n_\ell=\bigl\lceil{N_\ell}^{1/2}\bigr\rceil
\end{align*}
for $\ell=0,\dots,L$. Consider an independent family 
\[
\bG = 
(G_{k,2^\ell,j})_{k=1,\dots,2^\ell,\, \ell=0,\dots,L,\, j=1,\dots,2n_\ell}
\]
of random vectors $G_{k,2^\ell,j}$
that are uniformly distributed on $(D^{(q)})^d$.
For $x\in\R$ we denote by $x\mod 1$ the fractional part of $x$, i.e., 
the real number $y\in \left[0,1\right[$ that satisfies $x-y\in\Z$.
For $k=1,\dots,2^\ell$, $\ell=0,\dots,L$, and $j_1,j_2=1,\dots,n_\ell$ 
define
\begin{align*}
U_{k,2^\ell,(j_1-1)n_\ell+j_2}
=
G_{k,2^\ell,j_1}+G_{k,2^\ell,j_2+n_\ell}+
2^{-(q+1)}\cdot \boldsymbol{1} \mod 1,
\end{align*}
where the fractional part is taken in each component separately
and $\boldsymbol{1}\in\R^d$ denotes the vector with all components equal 
to $1$. 
Let
\[
\bU_\ell = 
(U_{k,2^\ell,i})_{k=1,\dots,2^\ell, i=1,\dots,N_\ell}
\]
for $\ell=0,\dots,L$.
Obviously, the following holds true:
\begin{itemize}
\item[(I)] The family $(\bU_\ell)_{\ell=0,\dots,L}$ is independent.
\item[(II)] For all $\ell=0,\dots,L$ and
$i=1,\dots,N_\ell$ the 
family $(U_{k,2^\ell,i})_{k=1,\dots,2^\ell}$ is independent.
\end{itemize}
Let
\[
\bU_{\ell,i} = 
(U_{k,2^\ell,i})_{k=1,\dots,2^\ell}
\]
for $\ell=0,\dots,L$ and $i=1,\dots,N_\ell$.
Lemma~\ref{bahvalov} from Appendix~\ref{app} yields:
\begin{itemize}
\item[(III)] For every $\ell=0,\dots,L$ the family 
$(\bU_{\ell,i})_{i=1,\dots,N_\ell}$ is 
pairwise independent.
\item[(IV)] For all $k=1,\dots,2^\ell$, $\ell=0,\dots,L$, and 
$i=1,\dots,N_\ell$ the random vector $U_{k,2^\ell,i}$ is uniformly 
distributed on $(D^{(q)})^d$.
\end{itemize}

For $k=1,\dots,2^\ell$, $\ell=0,\dots,L$, and $i=1,\dots,N_\ell$ let
\begin{align*}
V_{k,2^\ell,i}
=2^{-\ell/2}\cdot \Phi^{-1}\left(U_{k,2^\ell,i}\right),
\end{align*}
where the function $\Phi^{-1}$ is applied to each of the components 
of the $d$-dimensional random vector $U_{k,2^\ell,i}$ separately,
cf.\ \eqref{eq10}, and
for $k=1,\dots,2^{\ell-1}$, $\ell=1,\dots,L$, and $i=1,\dots,N_\ell$ let
\begin{align*}
\tV_{k,2^{\ell-1},i}
= V_{2k,2^\ell,i} + V_{2k-1,2^\ell,i},
\end{align*}
cf.\ \eqref{eq2}.
Finally, for $\ell=0,\ldots,L$ and $i=1,\ldots,N_\ell$ the random element 
$X_{2^\ell,i}$ is given by the Euler scheme using the 
increments $(V_{k,2^\ell,i})_{k=1,\dots,2^\ell}$,
for $\ell=1,\ldots,L$ and $i=1,\ldots,N_\ell$ the random element 
$\tX_{2^{\ell-1},i}$ is given by the Euler scheme using the increments 
$(\tV_{k,2^{\ell-1},i})_{k=1,\dots,2^{\ell-1}}$,
and for $\ell=0$ and $i=1,\ldots,N_\ell$ we set  
$\tX_{2^{\ell-1},i}=0$.

Observe that (I) and (III) imply (i).
Furthermore, (II) and (IV) imply
\[
(V_{k,2^\ell,i})_{k=1,\dots,2^\ell} 
\stackrel{\mathrm{d}}{=}
(V^\bit_{k,2^\ell,q})_{k=1,\dots,2^\ell} 
\]
for $\ell = 0, \dots,L$ and $i=1,\dots, N_\ell$,
so that we also have (ii).

The cost of $A_{L,N,q}^{\stwo}$ can be bounded as in 
Sections~\ref{classicml} and \ref{section1}.
The number of calls to the random number generator
is the number of random bits needed to simulate the
distribution of $\bG$.
Therefore it is given by
$\sum_{\ell=0}^L r_\ell$ with
\begin{align*}
r_\ell=2n_\ell\cdot 2^\ell\cdot q\cdot d.
\end{align*}
We add that $r_\ell$ is the number
of calls to the random number generator on level 
$\ell$. Furthermore,
there exists a constant $c>1$ such that for all $L$, $N$, and $q$
the multilevel Euler algorithm satisfies
\begin{align*}
c^{-1} \cdot \sum_{\ell=0}^L \left(
2^\ell\cdot N_\ell
+
n_\ell\cdot 2^\ell\cdot q
\right)
\leq
\cost\bigl(A_{L,N,q}^{\stwo},\Lip_1\bigr)
\leq
c \cdot \sum_{\ell=0}^L \left(
2^\ell\cdot N_\ell
+
n_\ell\cdot 2^\ell\cdot q
\right).
\end{align*}

For $\eps \in {]0,1/2[}$ we consider the algorithm 
\[
A_\eps^{\stwo} = A_{L(\eps),N(\eps),q(\eps)}^{\stwo}
\]
with $L(\eps)$ and $N(\eps)$ as in Section~\ref{classicml}
and $q(\eps)$ as in Section~\ref{section1}.

\begin{thm}\label{theo3}
There exists a constant $c>1$ such that 
the random bit multilevel Euler algorithm $A_\eps^{\stwo}$ satisfies
\begin{align}\label{thm8error}
\ee\bigl(A_\eps^{\stwo},\Lip_1\bigr)
\leq c \cdot \eps
\end{align}
and
\begin{align*}
c^{-1} \cdot \eps^{-2} \cdot (\ln(\eps^{-1}))^3
\leq
\cost\bigl(A_\eps^{\stwo},\Lip_1\bigr)
\leq
c \cdot \eps^{-2} \cdot (\ln(\eps^{-1}))^3
\end{align*}
for every $\eps \in {]0,1/2[}$.
\end{thm}

\begin{proof}
The error bound follows directly from \eqref{error} and 
Theorem~\ref{theo2}.

The cost bounds follow 
directly from \eqref{g4} and the following 
observation.
Let $n_\ell(\eps) = \bigl\lceil N^{1/2}_\ell(\eps) \bigr\rceil$.
Then there exists a constant $c>1$ such that 
\begin{align*}
c^{-1} \cdot \eps^{-2} \cdot (\ln(\eps^{-1}))^{3/2}
\leq 
\sum_{\ell=0}^{L(\eps)}
n_\ell(\eps) \cdot 2^\ell
\leq 
c \cdot \eps^{-2} \cdot (\ln(\eps^{-1}))^{3/2}
\end{align*}
for all $\eps$.
\end{proof}

Up to a multiplicative constant the error and cost bounds for the
algorithm $A_\eps^{\stwo}$, which uses only random bits, and the
classical multilevel algorithm $A_\eps^\class$, 
which uses random numbers from $[0,1]$, coincide.
A proper comparison, based on lower bounds,
of the corresponding classes of randomized
algorithms is given in Section \ref{s53}.

\begin{rem}\label{cr2}
We sketch 
an algorithm that uses asymptotically fewer random bits 
than the algorithm $A_\eps^{\stwo}$,
but still has the desired 
distributional properties (i) and (ii).
However, the overall cost will not be improved.
The idea is to use the following variant of Lemma~\ref{bahvalov}.
Consider an independent family $(G_{i,j})_{i=1,2,\, j=1,\dots,n}$ of 
random variables that are uniformly distributed on $D^{(q)}$.
Then the family 
$(G_{i_1,1}+G_{i_2,2}+\dots+G_{i_n,n}+
2^{-(q+1)}\cdot(n-1) \mod 1)_{i_1=1,2,\dots,i_n=1,2}$
is pairwise independent with each random variable being uniformly 
distributed on $D^{(q)}$.
Proceeding similar to the construction of $A_{L,N,q}^{\stwo}$, 
we obtain an algorithm 
that only needs $d\cdot \sum_{\ell=0}^{L}2^\ell\cdot q\cdot 2\hat n_\ell$
random bits, where 
\[
\hat n_\ell=\lceil\log_2{N_\ell}\rceil
\]
if $N_\ell\geq 2$ and $\hat n_\ell=1/2$ if $N_\ell=1$.
With the choice of the parameters $L,N,q$ as above this number of
bits is of the order
$\eps^{-2} \cdot (\ln(\eps^{-1}))^2 \cdot \ln(\ln(\eps^{-1}))$.
See Appendix~\ref{apprem10} for the upper bound, and observe that
the number of bits needed on the highest level is already of this order.

Recall that the number of random bits needed for the algorithm 
$A_\eps^{\stwo}$ is of the order
$\eps^{-2} \cdot (\ln(\eps^{-1}))^{5/2}$,
see the proof of Theorem~\ref{theo3},
and the number of random bits needed for one path on the finest 
level is given by $d\cdot q\cdot 2^L$, which is of the order 
$\eps^{-2} \cdot (\ln(\eps^{-1}))^{2}$.
\end{rem}

\section{Lower Bounds for Random Bit Monte Carlo Algorithms}\label{bounds}

We derive a general lower bound for
random bit 
Monte Carlo algorithms, which is applied to the quadrature problem for 
SDEs. At first, we modify a basic setting from information-based complexity
in order to formally introduce the class of all
random bit 
Monte Carlo algorithms, cf.\ \citet{H18}.

\subsection{Random Bit Monte Carlo Algorithms}\label{ls1}

Consider a non-empty class $F$ of real-valued functions on a set
$\fX \neq \emptyset$ and a mapping 
\[
S \colon F \to \R,
\]
which is to be approximated. By assumption, a
random bit
Monte Carlo algorithm has access to the functions $f \in F$ via an
oracle (subroutine) that provides function values $f(x)$ for points
$x \in \fX$ and to an ideal generator for random bits.
This generator is modelled by the probability space 
$(\Omega,\mathfrak{A},P)$, where 
\[
\Omega = \{0,1\}^\N
\]
and where $\mathfrak{A}$ and $P$ denote the product
$\sigma$-algebra and the product measure of 
the power set of $\{0,1\}$ and the uniform distribution on
$\{0,1\}$, respectively. 
The cost per evaluation of $f \in F$
is modelled by a mapping 
\[
\cc \colon \fX \to \N\cup\{\infty\},
\]
see \citet[Sec.~2]{CDMGR09}. 

For any function $f \in F$, its sequential evaluation,
which may be interlaced with calls to the random bit generator,
is formally defined by a sequence of mappings
\[
\psi_1 \colon \{0,1\} \to \fX
\]
and
\[
\psi_\ell \colon \{0,1\}^\ell \times \R^{\ell-1} \to \fX 
\]
with $\ell \geq 2$. For every $f \in F$, the first step consists
of a call to the random bit generator, which yields a bit
$\omega_1 \in \{0,1\}$, followed by the evaluation of $f$ at 
$\psi_1(\omega_1) \in \fX$.
After $n$ steps $n$ bits $\omega_1,\dots,\omega_n \in \{0,1\}$
have been obtained, and the function values
\[
y_1 = f(\psi_1(\omega_1))
\]
and
\[
y_\ell =
f(\psi_\ell(\omega_1,\dots,\omega_\ell,y_1,\dots,y_{\ell-1}))
\]
with $\ell = 2,\dots,n$ are known. Trivially, 
$\omega \mapsto (y_1,\dots,y_n)$ yields a measurable mapping
from $\Omega$ to $\R^n$ for all $f \in F$ and $n \in \N$.
A decision to stop or to further
evaluate $f$ is made after each step. This is formally described by
a sequence of mappings
\[
\tau_\ell \colon \{0,1\}^\ell \times \R^{\ell} \to \{0,1\} 
\]
with $\ell \geq 1$, and the total number $n(\omega,f)$ of
evaluations of $f$ is given by
\[
n(\omega,f) = 
\min\{\ell\in\N\colon \tau_\ell(\omega_1,\dots,\omega_\ell,
y_1,\dots,y_\ell)=1\} \in \N \cup \{\infty\}
\]
for $\omega \in \Omega$.  
Trivially, $n(\cdot,f)$ is a measurable mapping from $\Omega$ to
$\N \cup \{\infty\}$. We assume that 
\[
\forall\, f \in F \colon
P(\{n(\cdot,f) < \infty\})=1.
\]
Finally, a sequence of mappings
\[
\phi_\ell \colon \{0,1\}^\ell \times \R^\ell \to \R
\]
with $\ell \geq 1$ yields the approximation
\[
\tS(\omega,f) = \phi_{n(\omega,f)}
(\omega_1,\dots,\omega_{n(\omega,f)}, y_1,\dots,y_{n(\omega,f)})
\]
to $S(f)$ for all
$\omega \in \Omega$ and $f \in F$ with $n(\omega,f)< \infty$.
Otherwise, i.e., if $n(\omega,f)= \infty$, we put
$\tS(\omega,f)=0$. Trivially, $\tS(\cdot,f)$ is a measurable
mapping for every $f \in F$.

The tuple $\cS = ((\psi_\ell)_\ell,(\tau_\ell)_\ell,(\phi_\ell)_\ell)$ 
will be considered as a
random bit
Monte Carlo algorithm, with algorithm 
being understood in a broad sense, cf.\ Section~\ref{seccomp},
and the class of all such algorithms is denoted by $\Sb$.
For $\cS \in \Sb$ the corresponding mapping $\tS: \Omega \times F \to \R$
will be called the input-output mapping induced by $\cS$.

For $\ell \in \N$ let $\pi_\ell\colon \Omega\to\{0,1\}^\ell$ be given by
\[
\pi_\ell(\omega)=(\omega_1,\dots,\omega_\ell).
\]
The information cost and the error for applying the 
algorithm $\cS \in \Sb$ to $f \in F$ are defined by
\[
\cost_\cc (\cS,f) = 
\E 
\left(\sum_{\ell=1}^{n(\cdot,f)}
\cc(\psi_\ell(\pi_\ell(\cdot),y_1,\dots,y_{\ell-1}))
\right)
\]
and
\[
\ee (\cS,f) = 
\left(
\E
|S(f) - \tS(\cdot,f)|^2
\right)^{1/2},
\]
respectively. 
Obviously, $\cost_\cc(\cS,f)$ is also an upper bound for
the expected number $\E(n(\cdot,f))$ of random bits that are used by $\cS$,
when applied to $f$.
The worst case information cost
and the worst case error of $\cS$ are defined by
\begin{equation}\label{g21}
\cost_\cc(\cS) = 
\sup_{f \in F} 
\cost_\cc (\cS,f)
\end{equation}
and
\begin{equation}\label{g22}
\ee(\cS) = \sup_{f \in F}
\ee (\cS,f),
\end{equation}
respectively. The $\eps$-complexity for the approximation
of $S$ by means of
random bit
Monte Carlo algorithms is defined by
\[
\cb_\cc (\eps) = \inf \{ \cost_\cc(\cS) \colon
\cS \in \Sb,\ \ee(\cS) \leq \eps\}
\]
for $\eps > 0$. 

If for every $\ell \geq 1$ none of the mappings 
$\psi_\ell$, $\tau_\ell$, and $\phi_\ell$ depends
on $\omega_1,\dots,\omega_\ell$, then $\cS$ is called a
deterministic algorithm, and the corresponding proper subclass
of $\Sb$ is denoted by $\Sd$. For $\cS \in \Sd$ the expectations
may be dropped in the definitions of the worst case information
cost and error, and the $\eps$-complexity for the approximation
of $S$ by means of deterministic algorithms is defined by
\[
\cd_\cc (\eps) = \inf \{ \cost_\cc(\cS) \colon
\cS \in \Sd,\ \ee(\cS) \leq \eps\}
\]
for $\eps > 0$.

Monte Carlo algorithms from the class $\Sb$ may only use random bits.
Let us drop this constraint and consider randomized algorithms
in general.  A randomized algorithm is defined by 
any probability space $(\Omega,\mathfrak A, P)$ and a mapping
$\cS \colon \Omega \to \Sd$, which induces an input-output mapping
$\widehat{S} \colon \Omega \times F \to \R$.
Roughly speaking, $(\Omega,\mathfrak A, P)$ is the computational
probability space, which may carry any kind of random number
generator. If for a fixed $\omega \in \Omega$ all
random numbers that may potentially be used in the computation
are fixed in advance, then the randomized algorithm is turned into
the deterministic algorithm $\cS(\omega)$ with input-output mapping
$\widehat{S} (\omega,\cdot)$. 
Under the appropriate measurability assumptions, 
the worst case error and the worst case information cost of a 
randomized algorithm $\cS$ are defined analogously to \eqref{g21} 
and \eqref{g22}.
The class of all randomized algorithms that satisfy the
measurability assumptions is denoted by $\Sr$, and
\[
\cra_\cc (\eps) = \inf \{ \cost_\cc(\cS) \colon
\cS \in \Sr,\ \ee(\cS) \leq \eps\}
\]
is the $\eps$-complexity for the approximation of $S$ by
means of randomized algorithms. 
We refer to, e.g., \citet[Sec.~2.3]{CDMGR09} for further details.
Obviously,
\begin{align}\label{g77}
\cra_\cc(\eps) \leq \cb_\cc(\eps) \leq \cd_\cc(\eps)
\end{align}
for every $\eps >0$.

\begin{rem}
Every algorithm $\cS \in \Sb$ is asking for 
exactly one new random bit prior to a new function evaluation.
In a refined model, this number of random bits could vary; 
observe that the latter is actually 
the case for the random bit multilevel algorithms that have been
studied in Sections~\ref{section1} and \ref{secbah}.
On the other hand, 
cost one should be charged
explicitly for every call
of the random bit generator in a refined model.
Formally an algorithm in this refined model
may be turned into an algorithm $\cS \in \Sb$ by additional calls of the 
random bit generator and additional evaluations of $f$ 
at a fixed element $y \in \fX$
with no impact on the further computation.
Observe that the expected cost for applying the original algorithm in the 
refined model to $f \in F$ is bounded from above by 
$\cc(y) \cdot \cost_\cc(\cS,f)$.

We stress that even this refined model disregards further details
of the real-number model, cf.\ Section~\ref{seccomp}. However,
for the purpose of establishing lower bounds with unspecified 
multiplicative constants, it suffices to study algorithms
$\cS \in \Sb$ and $\cost_\cc(\cS,f)$, i.e., the information cost,
in the non-trivial case that there exists an element $y \in \fX$
with $c(y) < \infty$.

The refined model is formally introduced in \citet{H18}. In this model
upper and lower bounds that separately take into account the 
number of random bits
and the information cost may be established in a natural
way. Furthermore, this framework allows to
consider
Monte Carlo algorithms
that only have access to generators for an arbitrary, but fixed
set $\mathcal{Q}$ of probability distributions.
Random bit Monte Carlo algorithms, as studied in the present
paper, correspond to the case $\mathcal{Q} = \{Q\}$ with $Q$ denoting 
the uniform distribution on $\{0,1\}$, while 
$\mathcal{Q} = \{Q\}$ with $Q$ denoting 
the uniform distribution on $[0,1]$ is most often considered
for quadrature problems.
\end{rem}

\subsection{A General Lower Bound}\label{ls2}

Let $\cS = ((\psi_\ell)_\ell,(\tau_\ell)_\ell,(\phi_\ell)_\ell) \in
\Sb$, and suppose that there exist $k,c \in \N$ with the following
properties:
For all $\omega \in \Omega$, $f \in F$, and $y_1,\dots,y_{k-1} \in
\R$ we have $n(\omega,f) = k$ and
\[
\sum_{\ell=1}^{k}
\cc(\psi_\ell(\omega_1,\dots,\omega_\ell,y_1,\dots,y_{\ell-1})) = c,
\]
i.e., $k$ random bits are used in any case and the information cost
is deterministic and does not depend on $f$.
Then $\cS$ can be emulated by $2^k$ deterministic algorithms, each
with information cost $c$ for every $f$,
plus a random choice among the outputs of the deterministic
algorithms. This observation can be generalized to any $\cS \in
\Sb$, and it leads to a lower bound for $\cb_\cc (\eps)$ in terms 
of $\cd_\cc (\eps)$.
Under the particular assumptions mentioned above
this lower bound has already been established in
\citet[Prop.~1]{MR2076605}.

\begin{lem}\label{l22}
For every algorithm $\cS \in \Sb$ with $\cost_\cc(\cS)<\infty$ 
there exists an algorithm $\cS^* \in \Sd$ such that
\[
\cost_\cc(\cS^*)
\leq 2 \cdot (\cost_\cc (\cS)+1) \cdot 4^{\cost_\cc (\cS)+1}
\]
and
\[
\ee(\cS^*) \leq 2 \cdot \ee(\cS).
\]
\end{lem}

\begin{proof}
Let $\cS = ((\psi_\ell)_\ell,(\tau_\ell)_\ell,(\phi_\ell)_\ell) \in
\Sb$ with input-output mapping $\tS$. We may assume that 
$\ee(\cS)<\infty$, in addition to $\cost_\cc(\cS) < \infty$.
Hence $\tS(\cdot,f)$ is integrable for every $f\in F$.
Moreover, let $n (\cdot,\cdot)$ be the corresponding total number of
evaluations. Put
\begin{align*}
k=\lceil\cost_\cc(\cS)\rceil\in\N.
\end{align*}
For $f\in F$ we define
\begin{align*}
A_f=\Big\{\omega\in\Omega\colon
\sum_{\ell=1}^{n(\omega,f)}
\cc(\psi_\ell(\pi_\ell(\omega),y_1,\dots,y_{\ell-1}))
\leq 2k
\Big\}
\end{align*}
and
\begin{align*}
\tS^*(f)
=
\E(\tS(\cdot,f)\,\vert\, A_f).
\end{align*}
Since $P(A_f)\geq 1/2$, we get the error bound
\begin{align*}
|S(f)-\tS^*(f)|
=|S(f)-\E(\tS(\cdot,f)\,\vert\, A_f)|
\leq \E\big(|S(f)-\tS(\cdot,f)| \,\big\vert\, A_f\big)
\leq 2 \cdot \ee (\cS,f)
\end{align*}
for the conditional expectation $\tS^*(f)$.

It remains to show that $\tS^*$ is the input-output mapping
of an algorithm $\cS^* \in \Sd$ with a cost bound as claimed.
For $\ell\in \N$ let $\widetilde{\pi}_\ell\colon \{0,1\}^\ell\to\Omega$ 
be given by
\begin{align*}
\widetilde{\pi}_\ell(\omega)
=(\omega,0,0,\dots).
\end{align*}
Since $\cc \geq 1$, we have $n(\omega,f) \leq 2k$ for $\omega \in
A_f$. Hereby it follows that 
\begin{align*}
1_{A_f}
=1_{A_f} \circ \widetilde{\pi}_{2k}\circ \pi_{2k}
\end{align*}
and
\begin{align*}
\tS(\cdot,f) \cdot 1_{A_f} =
\tS(\cdot,f) 
\circ \widetilde{\pi}_{2k}\circ \pi_{2k}
\cdot 1_{A_f} =
\bigl(\tS(\cdot,f) \cdot 1_{A_f} \bigr)
\circ \widetilde{\pi}_{2k}\circ \pi_{2k}.
\end{align*}
Furthermore, $\pi_{2k}$ is uniformly distributed on $\{0,1\}^{2k}$.
Hence we get
\begin{align*}
\tS^*(f)
=\frac{1}{|B_f|} \cdot
\sum_{\omega\in \{0,1\}^{2k}}
\tS(\widetilde{\pi}_{2k}(\omega),f)\cdot 1_{B_f}(\omega),
\end{align*}
where
\begin{align*}
B_f
=
\{\omega\in \{0,1\}^{2k}\colon \widetilde{\pi}_{2k}(\omega)\in A_f\}.
\end{align*}
For every $\omega\in \{0,1\}^{2k}$ and every $f \in F$ the
information cost of 
simultaneously computing
$1_{B_f}(\omega)$
and
$\tS(\widetilde{\pi}_{2k}(\omega),f)\cdot 1_{B_f}(\omega)$
is bounded by $2k$. We conclude that
there exists an algorithm $\cS^* \in \Sd$
with input-output mapping $\tS^*$
and $\cost_\cc(\cS^*) \leq 2k \cdot 4^{k}$.
\end{proof}

\begin{thm}\label{t22}
There exist universal constants $c_1,c_2 > 0$ such that
\[
c_1 \cdot \ln \left(\cd_\cc (2 \eps)\right) - c_2 \leq \cb_\cc (\eps)
\]
for every $\eps>0$ with $\cd_\cc (2 \eps)<\infty$.
Furthermore, $\cb_\cc (\eps)=\infty$ if
$\cd_\cc (2 \eps)=\infty$.
\end{thm}

\begin{proof}
Let $\cS \in \Sb$ with $\ee(\cS) \leq \eps$ and $\cost_\cc(\cS) <
\infty$. Choose $\cS^* \in \Sd$ according to Lemma~\ref{l22}, and
put $z = \cost_\cc(\cS) + 1$. We obtain
\[
\cd_\cc(2 \eps) \leq \cost_\cc(\cS^*) \leq
2 z \cdot 4^z \leq 2 \cdot \exp( (1+ \ln (4)) \cdot z),
\]
which yields the lower bound for $\cost_\cc(\cS)$ as claimed.
\end{proof}

In many cases the lower bound from Theorem~\ref{t22} is far from sharp, see, e.g., \citet{MR2076605} for finite-dimensional quadrature problems.
The lower bound is useful, however, for the one-dimensional case studied in Example~\ref{ex1}, where we do not have any regularity,
as well as for the quadrature problem for SDEs.

\begin{rem}
Suppose that $\fX$ is a Banach space, $F=\Lip_1$, and
$S$ is defined by integration with respect to a Borel probability measure
$\mu$ on $\fX$. It is well known that $\cd_\cc$ is determined by the 
quantization numbers of $\mu$ of order one, if 
$\cc=1$, see, e.g., \citet{CDMGR09}. This relation is exploited in 
the proof of Theorem~\ref{t23} below.
\end{rem}

We present a computational problem that is trivial when
random numbers may be used, but unsolvable if only random
bits are available.

\begin{exmp}\label{ex1}
Let $\fX = [0,1]$, let $F$ denote the class of all functions
$f \colon [0,1] \to \R$ that are constant $\lambda$-a.e., and let
\[
S(f) = \int_{[0,1]} f\,\mathrm d \lambda
\]
for $f \in F$. Consider $\cc=1$, which is most natural in the
present setting, and let $\eps > 0$.
Obviously $\cd_\cc(\eps) = \infty$, so that Theorem \ref{t22}
yields
\[
\cb_\cc(\eps) = \infty.
\]
On the other hand, $S(f) = f(\omega)$ for $\lambda$-a.e. $\omega
\in [0,1]$, and therefore
\[
\cra_\cc(\eps) = 1. 
\]

The same phenomenon is present for
Monte Carlo algorithms that only have access to any fixed set $\mathcal{Q}$ 
of probability distributions: For any such $\mathcal{Q}$ there
exists a computational problem that is trivial for the class $\Sr$,
but is unsolvable by means of
Monte Carlo
algorithms with generators only for the distributions from
$\mathcal{Q}$.
See \citet[Thm.~3.1]{H18}.
\end{exmp}

\subsection{A Lower Bound for Quadrature of SDEs}\label{s53}

We return to the quadrature problem for SDEs, see
Section~\ref{seccomp}. Hence we have
\[
\fX = C([0,1],\R^r)
\]
and 
\[
F = \Lip_1,
\]
and $S$ is given by \eqref{g20}. The function $\cc$ is given as
follows. 
For $\ell \in \N_0$ let $\fX_\ell$ denote the set of all
piecewise linear functions $x \in \fX$ with 
breakpoints $k/2^\ell$ for $k=0,\dots,2^\ell$, so that
the spaces $\fX_\ell$ form an increasing sequence of
finite-dimensional subspaces of $\fX$, whose union is dense in
$\fX$. Obviously, $\dim(\fX_\ell) = 2^\ell+1$, and $\cc$ is defined
by
\begin{equation}\label{g23}
\cc(x) = \inf \{ \dim(\fX_\ell) \colon x \in \fX_\ell\}
\end{equation}
for $x \in \fX$.
Compared to Sections~\ref{seccomp}--\ref{sec5}, we impose slightly 
stronger smoothness assumption as well as a non-degeneracy assumption
on the diffusion coefficient $b$ of the SDE, which in particular 
exclude pathological cases yielding a deterministic solution of the SDE.

\begin{thm}\label{t23}
Assume that $r=d$ and that
\begin{itemize}
\item
$b:\R^r\to \R^{r\times r}$ has bounded first and second
order partial derivatives and is of class $C^\infty$ in some
neighborhood of $x_0$,
\item
$\det b(x_0)\not=0$,
\end{itemize}
in addition to the Lipschitz continuity of $a:\R^r\to \R^r$.
Then there exist constants $c,\eps_0>0$ such that 
\[
\cb_\cc(\eps) \geq c \cdot \eps^{-2} 
\]
for every $\eps \in {]0,\eps_0]}$. 
\end{thm}

\begin{proof}
According to \citet[Thm.~1, Prop.~3]{CDMGR09} there exist
constants $c,\eps_0>0$ such that $\cd_1(\eps) \geq \exp(c \cdot
\eps^{-2})$ for every $\eps \in {]0,\eps_0]}$. 
Apply Theorem~\ref{t22}.
\end{proof}

We combine Theorems~\ref{theo3} and \ref{t23}. 

\begin{cor}\label{c1}
Under the assumptions from Theorem \ref{t23}
the following holds true:
There exist constants $c_1,c_2,\eps_0 >0$ such
that 
\[
c_1 \cdot \eps^{-2} \leq
\cb_\cc(\eps) \leq c_2 \cdot \eps^{-2} \cdot
\bigl(\ln(\eps^{-1})\bigr)^3
\]
for every $\eps \in {]0,\eps_0]}$, and 
the random bit multilevel Euler algorithm $A_\eps^\stwo$ that
achieves the upper bound is almost optimal.
\end{cor}

We stress that the lower bounds from Theorem \ref{t23}
and from Corollary \ref{c1}
hold even for $\cc=1$, i.e., in
the most generous cost model, where every $f \in \Lip_1$ 
may be evaluated at any point $x \in C([0,1],\R^r)$ at unit cost.
The key to prove the lower bound
is the number of random bits together with the number of 
function evaluations being used by any algorithm $\cS \in \Sb$.

Due to \citet[Thm.~11]{CDMGR09} we have the following lower bound 
for the class $\Sr$ 
under the assumptions from Theorem \ref{t23}:
There exist constants $c_1,\eps_0 >0$ such that 
\[
\cra_\cc(\eps) \geq c_1 \cdot \eps^{-2}
\]
for every $\eps \in {]0,\eps_0]}$.
Obviously the latter estimate together with \eqref{g77} yields a second proof of Theorem~\ref{t23}.
We stress that this lower bound is no longer valid for $\cc=1$,
i.e., it does not suffice to only consider the number of
function evaluations being used by any algorithm $\cS \in \Sr$.
The lower bound is valid, however, for every $\cc$ according to
\eqref{g23} based on any increasing sequence of finite-dimensional
subspaces $\fX_\ell \subset \fX$.

While the asymptotic behavior of $\cb_\cc(\eps)$ and
$\cra_\cc(\eps)$ for $\eps \to 0$ is not known exactly, 
their ratio is bounded by a multiple of
$(\ln(\eps^{-1}))^3$. 
In particular, we do not know
whether random bits are as powerful as random numbers
for the quadrature problem under investigation,
but random bits are at least almost as powerful as random numbers.
Note that $\cb_\cc(\eps)$ is dramatically smaller than $\cd_1(\eps)$.

\appendix

\section{Bakhvalov's Trick}\label{app}

\begin{lem}\label{bahvalov}
Let $q\in\N$ and $n\in\N$.
Consider an independent family $(G_j)_{j=1,\dots,2n}$ of random
variables that are uniformly distributed on $D^{(q)}$.
Then the family 
$(G_{j_1}+ G_{j_2+n}+2^{-(q+1)} \mod 1)_{j_1,j_2=1,\dots,n}$
is pairwise independent with each random variable being uniformly 
distributed on $D^{(q)}$.
\end{lem}

\begin{proof}
Let $G$ be uniformly distributed on $D^{(q)}$.
Inductively, we get 
\begin{align}\label{appeq1}
\forall\, z\in 2^{-q}\cdot \N
\colon
z+G\mod 1
\text{ is uniformly distributed on }D^{(q)}.
\end{align} 
Furthermore,
\begin{align}\label{appeq2}
G+2^{-(q+1)}\in 2^{-q}\cdot \N.
\end{align}

Let $j_1,j_2\in\{1,\dots,n\}$.
Conditioned on $G_{j_1}$
the random variable $G_{j_2+n}$
is uniformly distributed on $D^{(q)}$.
Combining this with
\eqref{appeq1}
and
\eqref{appeq2}
shows that   
conditioned on $G_{j_1}$
the random variable $G_{j_1}+G_{j_2+n}+2^{-(q+1)}\mod 1$
is uniformly distributed on $D^{(q)}$.
Hence the random variable $G_{j_1}+G_{j_2+n}+2^{-(q+1)}\mod 1$
is uniformly distributed on $D^{(q)}$.

Let $i_1,i_2\in\{1,\dots,n\}$ such that $(i_1,i_2)\neq (j_1,j_2)$.
Next we show that $G_{j_1}+ G_{j_2+n}+2^{-(q+1)} \mod 1$
and $G_{i_1}+ G_{i_2+n}+2^{-(q+1)} \mod 1$
are independent.
The case $i_1\neq j_1$ and $i_2\neq j_2$ is trivial.
Thus without loss of generality we may assume $i_1=j_1$ and $i_2\neq j_2$.
Conditioned on $G_{j_1}$ the random variables $G_{i_2+n}$
and $G_{j_2+n}$ are independent and each uniformly distributed 
on $D^{(q)}$. Combining this with
\eqref{appeq1}
and
\eqref{appeq2}
shows that   
conditioned on $G_{j_1}$
the random variables
$G_{j_1}+G_{j_2+n}+2^{-(q+1)}\mod 1$
and
$G_{i_1}+G_{i_2+n}+2^{-(q+1)}\mod 1$
are independent and each uniformly distributed on $D^{(q)}$.
Hence we get the independence of the random variables
$G_{j_1}+G_{j_2+n}+2^{-(q+1)}\mod 1$
and
$G_{i_1}+G_{i_2+n}+2^{-(q+1)}\mod 1$.
\end{proof}

\section{An Estimate used in Remark~\ref{cr2}}\label{apprem10}

Observe that 
\[
\sum_{\ell=1}^L 2^\ell \cdot (L-\ell) = 2^{L+1}-2(L+1) \leq 2^{L+1}.
\]
Hence there exist $\eps_0\in {]0,1/2[}$ and 
$c_1,c_2,c_3>0$ such that for every $\eps\in {]0,\eps_0[}$ we have
\begin{align*}
\sum_{\ell=0}^{L(\eps)}2^\ell\cdot 
2\hat n_\ell(\eps)
&\leq
c_1 \cdot \sum_{\ell=1}^{L(\eps)}2^\ell\cdot
\log_2(L(\eps)\cdot 2^{-\ell}\cdot \ell \cdot \eps^{-2})
\\&\leq
c_1\cdot 
\sum_{\ell=1}^{L(\eps)}2^\ell \cdot
\left( L(\eps)-\ell+\log_2(L(\eps)^2)+\log_2(\eps^{-2})-L(\eps)\right)
\\&\leq
c_1\cdot 
\sum_{\ell=1}^{L(\eps)}2^\ell \cdot
\left( L(\eps)-\ell+c_2\cdot\ln(\ln(\eps^{-1}))\right)
\\&\leq
c_1\cdot 
\left( 2^{L(\eps)+1}+c_2\cdot \ln(\ln(\eps^{-1}))\cdot 2^{L(\eps)+1}\right)
\\&\leq
c_3\cdot 
\eps^{-2}\cdot\ln(\eps^{-1}) \cdot  \ln(\ln(\eps^{-1})).
\end{align*}

\section*{Acknowledgment}
The authors are grateful to Steffen Omland for many valuable
discussions and contributions at an early stage of this project.
We thank Erich Novak for pointing out Example \ref{ex1} to us, and
Stefan Heinrich and Holger Stroot for valuable discussions and comments.

Mike Giles was partially supported by the UK Engineering and Physical
Science Research Council (EPSRC) through the ICONIC Programme Grant,
EP/P020720/1.
Lukas Mayer was supported by the Deutsche Forschungsgemeinschaft
(DFG) within the RTG 1932 
`Stochastic Models for Innovations in the Engineering Sciences'.

\bibliographystyle{abbrvnat}
\bibliography{bib}

\end{document}